\documentclass[11pt]{amsart}  %<<<1
\usepackage{amssymb,amsmath,amsthm,amscd}

\usepackage{graphicx}
\usepackage{lscape}
\numberwithin{equation}{section}

\newtheoremstyle{fancy1}{10pt}{10pt}{\itshape}{12pt}{\textsc\bgroup}{.\egroup}{8pt}{
}
\newtheoremstyle{fancy2}{10pt}{10pt}{}{12pt}{\itshape}{.}{8pt}{ }

\theoremstyle{fancy1}

\newtheorem{lem}[equation]{Lemma}

\newtheorem*{thm*}{Theorem}

\newtheorem*{main*}{Theorem}

\newtheorem*{cor*}{Corollary}
\newtheorem*{prop*}{Proposition}

\newtheorem*{problem*}{Problem}

\theoremstyle{fancy2}

\newtheorem*{rems*}{Remarks}

\newtheorem*{rem*}{Remark}

\newtheorem*{example*}{Example}

\newcommand{\cref}[1]{Corollary~\ref{#1}}

\newcommand{\lref}[1]{Lemma~\ref{#1}}

\newcommand{\CP}{\mathbb{C\mkern1mu P}}

\newcommand{\Sph}{\mathbb{S}}
\newcommand{\Disc}{\mathbb{D}}

\newcommand{\R}{{\mathbb{R}}}
\newcommand{\Z}{{\mathbb{Z}}}

\newcommand{\QH}{{\mathbb{H}}}

\newcommand{\G}{\ensuremath{\operatorname{G}}}

\newcommand{\Sp}{\ensuremath{\operatorname{Sp}}}

\newcommand{\SU}{\ensuremath{\operatorname{SU}}}

\renewcommand{\S}{\ensuremath{\operatorname{S}}}

\newcommand{\K}{\ensuremath{\operatorname{K}}}

\newcommand{\fg}{{\mathfrak{g}}}
\newcommand{\fk}{{\mathfrak{k}}}
\newcommand{\fh}{{\mathfrak{h}}}
\newcommand{\fm}{{\mathfrak{m}}}
\newcommand{\fn}{{\mathfrak{n}}}

\newcommand{\fp}{{\mathfrak{p}}}

\def\con#1=#2(#3){#1 \equiv #2 \bmod{#3}}

\newcommand{\ml}{\langle}                     % Riemannian metric (left )
\newcommand{\mr}{\rangle}                    % Riemannian metric (right)

\newcommand{\diag}{\ensuremath{\operatorname{diag}}}

\newcommand{\Ad}{\ensuremath{\operatorname{Ad}}}

\DeclareMathOperator{\Id}{Id}

\DeclareMathOperator{\spam}{span}

\newcommand{\Kpm}{K^{\scriptscriptstyle{\pm}}}
\newcommand{\Kp}{K^{\scriptscriptstyle{+}}}
\newcommand{\Km}{K^{\scriptscriptstyle{-}}}

\newcommand{\no}{\noindent}

\newcommand{\coo}{{cohomogeneity one }}
\newcommand{\com}{{cohomogeneity one manifold }}
\newcommand{\coms}{{cohomogeneity one manifolds }}
\newcommand{\coa}{{cohomogeneity one action }}
\newcommand{\coas}{{cohomogeneity one actions }}

\begin{document}

\title{Seven dimensional  cohomogeneity one manifolds with nonnegative curvature}

\author{Luigi Verdiani}
\address{University of Firenze}
\email{verdiani@math.unifi.it}
\author{Wolfgang Ziller}
\address{University of Pennsylvania}
\email{wziller@math.upenn.edu}
\thanks{ The first named author was supported by the University of Pennsylvania
 and would like to thank the Institute for its
hospitality.   The second named author was supported by a grant from
the National Science Foundation.}

\maketitle

\begin{abstract}\noindent
We show that a certain family of cohomogeneity one manifolds does not admit an invariant metric of nonnegative sectional curvature, unless it admits one with positive curvature. As a consequence, the classification of nonnegatively curved
cohomogeneity one manifolds in dimension 7 is reduced to only one
further family of candidates
\end{abstract}

%-------------- Article Text--------------------
\bigskip

A group action is called a \coa if its generic orbits are hypersurfaces. Such actions have been used frequently to construct examples of various types, Einstein metrics, soliton metrics, metrics with positive or non-negatively curvature, and metrics with special holonomy, see e.g.  \cite{DW, FH, GKS, GZ,GVZ, KS} for a selection of such results. The advantage of such metrics is that geometric problems are reduced to studying its behavior along a fixed geodesic $c(t)$ normal to all orbits.

\smallskip

In \cite{Ho} one finds a classification of  \coas on $n$-dimensional simply connected compact manifolds of dimensions $n\le 7$. As a consequence, he showed that such manifolds admit an invariant metric of nonnegative curvature, unless it is a member of a certain family of Brieskorn varieties $B_d$,  or a member of two 7-dimensional families $E_{p,q}$ and $E_{p,q}^*$. In \cite{GVWZ} it was shown that  the Brieskorn varieties $B_d$ do not admit an invariant metric of nonnegative curvature for $d\ge 3$, while $B_1=S^7$ and $B_2=T^1S^4$ admit nonnegatively curved invariant metrics. For the remaining two families, the questions remained open.

\smallskip

To describe these two families, we recall the structure of a \coo manifold. Let $G$ be a Lie group with subgroups $K_\pm$ and $H\subset K_\pm$ such that $K_\pm/H$ is a sphere. Then $K_\pm$ acts linearly on a disc $D_\pm$ and $M=G \times_{K_-} D_- \cup_{G/H} G \times_{K_+}  D_+$ defines a \com under the left action by $G$.
In terms of this language, $E_{p,q}$ is defined by

   $$G=\S^3\times \S^3,\ K_-=\{(e^{ip\theta},e^{iq\theta})\mid \theta\in\R\},\ K_+=\{(a,a)\mid a\in S^3\}\cdot H,  \text{ and } H\simeq\Z_2=\{(\pm 1,1)\}$$
 where we assume that  $\gcd(p,q)=1$ and $p$ even.

  \smallskip

 The special case of $E_{p,q}$ with $|p\pm q|=1$ is an Eschenburg space, which admits  an invariant metric with positive curvature, see \cite{Es,Zi}. In this paper  we show:

\begin{main*}
The cohomogeneity one manifold $E_{p,q}$ does not admit an $\S^3\times\S^3$ invariant metric with nonnegative curvature if $|p\pm q|>1$.
\end{main*}

The second  family  $E_{p,q}^*$ is determined by the same groups $G$ and $ K_-$, but $H$ is trivial and $K_+$ is connected.
The special case of $E_{1,1}^*$ is diffeomorphic to $ \CP^2\times \Sph^3$, and the product metric is invariant under $G$ and has   nonnegative curvature. For the remaining manifolds, we do not know if they admit invariant metrics with nonnegative curvature.

 \smallskip
Altogether we obtain:
\begin{cor*}
A $7$-dimensional cohomogeneity one manifold admits an invariant metric with nonnegative curvature, unless it is a Brieskorn variety $B_d^7$ with $d>2$, or $E_{p,q}$ with $|p\pm q|>1$, or possibly $E_{p,q}^*$ with $(p,q)\ne (\pm1,\pm1)$.
\end{cor*}

 \smallskip

 We now describe how to obtain the obstruction. Let $c(t)$ be a normal geodesic  perpendicular to all orbits. It meets the singular orbits at time $t=n L,\ n\in\Z$ with $c(2nL)\in G/K_-$ and $c((2n+1)L)\in G/K_+$. It is sufficient to describe the metric along the geodesic $c$, which is determined by a collection of functions describing the homogenous metric on the principal orbits $G/H$. At $t=n L$, the functions must satisfy certain smoothness conditions in order for the metric to extend smoothly over the singular orbits. The Weyl group, i.e. the stabilizer of the normal geodesic, implies some of these conditions. In particular, all functions must be even at  $t=(2n+1) L$. At $t=0$ and $t=2L$ the stabilizer groups are the same, and hence the smoothness conditions as well. But in this case they are not implied by the action of the Weyl group and we use the methods in \cite{GVZ} to determine them. In particular, they depend on $p,q$ and $|p\pm q|\ge 3$ implies that some functions are odd. We then use the concavity of virtual Jacobi fields developed in \cite{VZ1} in order to show that one of the action fields $Z^*,\ Z\in\fg$ is parallel on $[0,2L]$. This eventually leads to a contradiction by showing that the length of another action  field is constant on $[0,2L]$, but on the other hand must vanish at $t=L$ due to the smoothness conditions. For the family $E_{pq}^*$ the smoothness conditions at $t=0$ change and imply that all functions are even, and thus the proof breaks down. On the other hand, the method still implies strong restrictions on  a possible metric with nonnegative curvature.

%%%%%%%%%%%%%%%%%%%%%%%%%%%%%%
%%%%%%%%%%%%%%%%%%%%%%%%%%%%%%%
\section{Preliminaries}%%%%%%%%%%%
%%%%%%%%%%%%%%%%%%%%%%%%%%%%%%
%%%%%%%%%%%%%%%%%%%%%%%%%%%%%%%

\bigskip

For a general reference for this section see e.g.  \cite{AA,Zi}.
A  compact cohomogeneity one manifold is  the union of two
homogeneous disc bundles. Given  compact Lie groups $H,\, \Km ,\,
\Kp$ and $\G$ with inclusions $H\subset \Kpm \subset G$ satisfying
$\Kpm/H=\Sph^{\ell_\pm}$, the transitive action of $\Kpm$ on
  $\Sph^{\ell_\pm}$ extends to a linear action on the disc $\Disc^{{\ell_\pm}+1} $.
We can thus define
$M=G\times_{\Km}\Disc^{{\ell_-}+1}\cup G\times_{\Kp}\Disc^{{\ell_+}+1}$
glued along the boundary $ \partial (G\times_{\Kpm}\Disc^{\ell_\pm+1})=G\times_{\Kpm}\Kpm/H=G/H$
via the identity. $G$  acts on $M$ on each half via left action in the first component. This action has principal isotropy group $H$ and singular isotropy groups $\Kpm$.
 One possible description of a cohomogeneity one manifold is thus
simply in terms of the Lie groups $H\subset \{\Km , \Kp\}\subset G$.

\smallskip

A $G$ invariant metric  is determined by its restriction
to  a geodesic $c$  normal to all orbits.
 At the  points $c(t)$ which are regular with respect to the action of $G$,
 the isotropy is constant  equal to  $H$.
We fix a biinvariant inner product $Q$ on the Lie algebra $\fg$
 and choose a $Q$-orthogonal splitting $\fg=\fh\oplus\fh^\perp$.  At a regular point we identify
 the tangent space to the orbit $G/H$, i.e.   $\dot{c}^\perp\subset T_{c(t)}M$,  with $\fh^\perp$
 via action fields: $X\in\fh^\perp\to
X^*(c(t))$. $H$ acts on $\fh^\perp$ via the adjoint representation
and a $G$ invariant metric on $G/H$ is described by an $\Ad(H)$
invariant inner product on $\fh^\perp$. Along $c$ the metric on $M$
is thus given  by $g=dt^2+g_t$ with $g_t$ a one parameter family of $\Ad_H$ invariant metrics on $\fh^\perp$. These are described by endomorphisms
$$P_t\colon\fh^\perp\to\fh^\perp\ \text{ where } \ g_t(X^*,Y^*)_{c(t)}=Q(P_tX,Y), \ \text{ for } \ X,Y\in\fh^\perp.$$

These endomorphisms $P_t$ commute with the action of $\Ad_H$ since the metric is  $\Ad_H$ invariant and we thus choose a fixed $Q$-orthogonal splitting
$$
\fh^\perp=\fn_0\oplus\fn_1\oplus\ldots\oplus\fn_r.
$$
where $\Ad_H$ acts trivially on $\fn_0$ and irreducibly on $\fn_i, i>0$. The metric $P_t$ is arbitrary on $\fn_0$, and a multiple of $\Id$ on $\fn_i, i>0$. Furthermore, $\fn_i$ and $\fn_j$ are orthogonal if the representations of $\Ad_H$ are inequivalent. If they are equivalent, inner products are described by $1,2$ or $4$ functions, depending on wether the equivalent representations are orthogonal, complex or quaternionic.

We choose the normal geodesic such that $c(t),\ 0\le t\le L$ is a minimizing geodesic between the two singular orbits. Thus $c(2n \cdot L)\in G/K_-$ and $c((2n+1) \cdot L)\in G/K_+$.
We  choose a fixed $Q$ orthonormal basis $X_i$ of $\fh^\perp$, adapted to the above decomposition The metric $P_t$ is then described by the functions $f_{ij}(t)=g_t(X_i^*,X_j^*)_{c(t)}$ defined for $t\in\R$. These functions must  satisfy certain smoothness conditions at  $c(n\cdot L),\ n\in\Z$ in order for the metric to extend smoothly across the singular orbits.
 Some of these smoothness conditions are implied by the action of the Weyl group. Recall that the Weyl group $W$
is by definition the stabilizer of the geodesic $c$ modulo its
kernel, which by  construction is equal to $H$.  If $M$ is compact, $W$ is a dihedral subgroup of $N(H)/H$  and is generated by involutions $w_{\pm} \in
W$ with $w_-(\dot c(0))=-\dot c(0)$ and
$w_+(\dot c(L))=-\dot c(L)$. Thus $w_+w_-$ is a translation by
$2L$, and has order $|W|/2$ when $W$ is finite. Since $K_-$ acts linearly on the slice $D_-$, and the identity component  $(\K_-)_0$ acts transitively on the unit sphere in $D_-$,  the involution
$w_-$ can  be represented  uniquely as the
 element $a\in(\K_-)_0$ mod $H$ with $a\dot c(0)=-\dot c(0)$, where $\dot c(0)\in D_- $, and similarly for $w_+$.
  Note that $W$ is finite if and only if $c$ is a closed
geodesic, and in that case the order $|W|$ is the number of minimal
geodesic segments intersecting the regular part.  Note also that any
non-principal isotropy group along $c$ is of the form $w K_\pm
w^{-1}$ for some $w \in N(H)$ representing an element of $W$.

\smallskip

At a singular point $t_0=nL$, with stabalizer group $K$, we define a  $Q$-orthogonal decomposition:
 $$\fg=\fk\oplus \fm  , \quad  \fk=\fh\oplus \fp \ \text{ and thus }\ \fh^\perp=\fp \oplus \fm . $$
  Here $\fm$ can be viewed as the tangent space to the singular orbit $G/K$ at
  $c(t_0)$. The slice $V$, i.e. the vector space normal to $G/K$ at
  $c(t_0)$, can be identified with  $\dot c(t_0)\oplus \fp$.
   For this, we send $X\in \fp$ to $\bar X= \lim_{t\to t_0} X^*(c(t))/t$. Note that since $K$ preserves the slice $V$ and acts linearly on it, we  have $X^*(c(t))=t\bar X\in V$.

At $t=t_0$ the slice is orthogonal to the orbit, but not in general at nearby points.
 $K$ acts via the
isotropy action $\Ad(K)_{| \fm}$ of $G/K$ on $\fm$ and via the slice
representation on $V$.  Thus $w_-$ also acts on $V$ and the tangent space $\fm$, well defined up to the action of $H$, and relates the functions. If, e.g., $w_-$ preserves one of the modules $\fn_i$, then the  function $f$ defined by ${P_t}_{|\fn_i}=f\Id$ must be even (assuming for simplicity that $t_0=0$) since $w_-(c(t)=c(-t)$ and $(\Ad(w)(X))^*(c(t))=X^*(c(-t))$. Hence $f(t)=f(-t)$.

 \bigskip

 \section{Obstructions}

 \bigskip

 We now discuss the \coms $E_{pq}$.

 \smallskip
 Recall that $G=\Sp(1)\Sp(1)$ and
 $$K_-=\S^1_{p,q}=(e^{ipt},e^{iqt}), \  K_+=\{(q,q)\mid q\in\Sp(1)\}\cdot H, \text{ and }  H=\{(1,\pm 1)\}.$$
 with $p$ even  and $\gcd(p,q)=1$, and hence $q$ odd.
  If $|p\pm q|=1$ the manifold is an Eschenburg space, see \cite{Es}, which admits a cohomogeneity one metric whose group diagram is as above and an invariant metric of positive curvature. We will thus assume that  $|p\pm q|\ne 1$, and hence $|p\pm q|\ge 3$. In particular, we also have $p\ne 0,\ q\ne 0$.

As in the previous Section, we assume that the normal geodesic $c$ connecting the two singular orbits is parameterized on the interval $[0,L]$. We first determine the Weyl group $W$. For $t=0$, the group $K_-$ acts via rotation on the slice $D_-$ with stabilizer group $H$ at $\dot c(0)$, i.e.  $(e^{ip\pi},e^{iq\pi})\in H$. Since $p$ is even and $q$ odd, $w_-=(e^{ip\pi/2},e^{iq\pi/2})$ is one of $ (\pm 1,\pm i) \mod H$. For $t=L$, $K_+^0=\Sp(1)$ acts via left multiplication on $D_+\simeq\QH$ and thus $w_+=(-1,-1)$. Since $w_+$ lies in the center of $G$, it acts trivially on  $\dot c(L)^\perp\subset T_{c(L)}M$ and hence all functions are even at $t=L$.
 Since  $w_+(c(0))=c(2L)$  we also have  $G_{c(2L)}=w_+(K_-)w_+^{-1}=K_-$ and in fact $G_{c(2n L)}=K_-$. On the other hand, $G_{c(3L)}=w_-(K_+)w_-^{-1}$ is different from $G_{c(L)}=K_+$. Notice also that $W=\Z_2\oplus\Z_2$ and hence $c(4L)=c(0)$.

As a basis of $\fh^\perp=\fg$ we choose $X_1=(i,0),\ X_2=(j,0),\
X_3=(k,0)$ and $Y_1=(i,0),\ Y_2=(j,0),\
Y_3=(k,0)$. Since $H$ lies in the center of $G$, there are no restrictions for the metric $P_t$ on $\fh^\perp$.
We define

$$f_i=g( X_i^*,X_i^*) ,\quad g_i=g( Y_i^*,Y_i^*),\quad h_i=g( X_i^*,Y_i^*)$$
$$f_{ij}=g (X_i^*,X_j^*),\quad g_{ij}=g( Y_i^*,Y_j^*),\quad h_{ij}=g( X_i^*,Y_j^*)\quad \text{ for } i\neq j$$
\smallskip

Notice that at $t=0$ the only vanishing Killing field is $pX_1^*+qY_1^*$, whereas at $t=L$, only $X_i^*+Y_i^*$, $i=1,2,3$ vanish.
For the proof, we will only need the  following smoothness conditions at $t=0$.

\begin{lem}\label{smooth}
If the metric is smooth at $t=0$, then the following hold:
\begin{itemize}
\item[(a)] $X_1^*,Y_1^*,X_3^*,Y_3^*$ are orthogonal to $X_2^*$ and $Y_2^*$,
\item[(b)] $f_2 \text{ and }  g_2 \text{ are even, and } h_2=t^{ k}\phi_1(t^2)$, where $k=\min\{|p-q|,|p+q|\}$,
\item[(c)] $g(X_2^*,p X_1^*+q Y_1^* )=t^{ |p|+2}\phi_2(t^2),\ g(Y_2^*,p X_1^*+q Y_1^* )=t^{ |q|+2}\phi_3(t^2).$
\end{itemize}
where $\phi_i$ are smooth functions.
\end{lem}
\begin{proof}
The proof is similar to the proof of Theorem 6.1 in \cite{GVZ}, see also \cite{VZ2}. The decomposition of $\fh^\perp=\fg$ at $t=0$ is given as follows. $\fp$ is spanned by $p X_1+q Y_1$, and $\fm_0$ is spanned by $qX_1-pY_1$ on which $K_-$ acts trivially. Furthermore on $\fm_1=\spam\{X_2,X_3\}$, the stabilizer group $K_-$ acts by $R(2p\theta)$, and on  $\fm_2=\spam\{Y_2,Y_3\}$ by $R(2q\theta)$, where $R(\theta)$ is the standard rotation on $\R^2$. On the slice $K_-$ acts as $R(2\theta)$ since $H\subset K_-$ is the ineffective kernel of the action, and $H$ has order $2$.

Part (a) follows from the fact that the metric is invariant under the action of $K_-$  at $t=0$. Indeed,  since $p\ne \pm q$, all irreducible sub representation of $K_-$ on $\fm=\fm_0\oplus\fm_1\oplus\fm_2$ are inequivalent. Furthermore, $\fp$ is orthogonal to $\fm$ at $t=0$. Thus the claim follows from Schur's Lemma.

Lemma 6.4 in \cite{GVZ} implies that $f_2-f_3=t^{2p}\phi_1(t^2)$ and $ f_2+f_3=\phi_2(t^2)$ for some smooth functions $\phi_i$ and hence $f_2$ is even. Similarly  for $g_2$. For $h_2$ we use Lemma 6.5 in \cite{GVZ}. It implies that $h_2+h_3=t^{|p-q|}\phi_1(t^2)$ and $h_2-h_3=t^{|p+q|}\phi_2(t^2)$ if $p$ and $q$ have the same sign and $h_2+h_3=t^{|p+q|}\phi_1(t^2)$ and $h_2-h_3=t^{|p-q|}\phi_2(t^2)$ if their signs are opposite. This implies part (b).

For part (c) the proof of  \cite{GVZ} Lemma 6.5 needs to be modified as follows, see also \cite{VZ2}. Let $e_1,\ e_2$ be a basis of the slice $D_-$, orthonormal in $g_{|D_-}$, where $e_1=\dot c(0),\ e_2=Z\in\fp$ and hence $Z=pX_1+qY_1$. Furthermore,  $K_-$ acts by $R(2\theta)$ on $D_-$. On the 4 dimensional space $D_-\oplus\fm_1$ the stabilizer group $K_-$ acts by $\diag(R(2\theta),R(2p\theta))$. The metric is given by $g(e_i,e_j)=\delta_{ij}$ and  $g(e_1,X_2)=g(e_1,X_3)=0$ since the geodesic is orthogonal to all orbits. Furthermore, set $r(t)=g(e_2, X_2^*)$ and $ s(t)=g(e_2, X_3^*)$. By \cite{VZ2}  Lemma 3.4 (b), we see that $r(t)=t^{|p+1|}\phi_1(t^2)$ and $r(t)=t^{|p-1|}\phi_2(t^2)$ and hence $r(t)=t^{|p|+1}\phi(t^2)$. Recall that for $Z\in\fp$ we have $Z^*(c(t))=te_2$. Hence $g( Z^*, X_2^*) =  g(te_2, X_2^*)=t\; r(t)=t^{|p|+2}\phi(t^2)$. Similarly for $g( Z^*, Y_2^*)$.
\end{proof}

These conditions hold at $t=0$, and by the above at $t=2L$ as well. For us, the important consequence is that $h_2$ is an odd function with $h_2'(0)=0$, a property not implied by the Weyl group.

\begin{lem}\label{parallel}
If  $E_{p,q}$ with $|p\pm q|\ge 3$ carries a metric with nonnegative curvature, then $X_2^*-Y_2^*$ is parallel on $[0,2L]$.
\end{lem}
\begin{proof}
 We will use  the results in \cite{VZ1}. Recall that a vector space $V$ of Jacobi fields is self adjoint along a geodesic $c$ if $\dim V=\dim M-1$ and $\ml J_1,J_2'\mr=\ml J_1',J_2\mr$ for all $J_1,J_2\in V$.   Proposition 3.2 in \cite{VZ1} states that a Jacobi field $X$ which belongs to a self adjoint family $V$ of Jacobi fields is parallel on an interval $[t_0,t_1]$, if the following conditions are satisfied:

\begin{enumerate}
 \item[(a)] $g_{c(t)}(X,X)\neq 0$, $g_{c(t)}(X,X)'=0$ for $t=t_0$ and $t=t_1$,\vspace{4pt}
 \item[(b)] If $Y\in V$ and $g( X(t_1),Y(t_1))=0$ then $g( X(t_0),Y(t_0))=0$,\vspace{4pt}
 \item[(c)] If $Y\in V$ and $Y(t)=0$ for some $t\in (t_0,t_1)$ then $g( X(t_0),Y(t_0))
 =0$,\vspace{4pt}
 \item[(d)] If $Y(t_0)=0$, then $g( X'(t_0),Y'(t_0))=0$.
\end{enumerate}

 \smallskip

 Since the restriction of the Killing vector fields $Z^*,\ Z\in \fg$ to the geodesic $c$ are Jacobi fields, a natural self adjoint family is given by $V=\{Z^*\mid Z\in \fh^\perp\}$. For simplicity we use the same letter for $Z\in\fg$ and the action field $Z^*$ on $M$, and denote the covariant derivative $\nabla_{\dot c}Z^*(c(t))$ by $Z'$.

Let $X=a X_2+b Y_2$ be the vector  in $\spam\{X_2,Y_2\}$ such that $g_{c(0)}(X,X_2+Y_2)=0$.   We first show that $X$ is parallel on $[0,2L]$, and verify the conditions above one at a time.

\begin{enumerate}
\item[(a)] $X$ does not vanish in $[0,2L]$ since the only Jacobi fields that vanish are $pX_1+qY_1$ at $t=0$ and $t=2L$, $X_i+Y_i$ at $t=L$, and  by assumption $X\ne X_2+Y_2$.
\no For the derivative we have $g(X,X)'=a^2f_2'+2abh_2'+b^2g_2'$. But since $|p\pm q|\ge 3$, \lref{smooth} (b) implies that $f_2', h_2'$ and $g_2'$ vanish at $t=0$, and $t=2L$ as well. \vspace{4pt}
\item[(b)] Recall that $G_{c(0)}=G_{c(2L)}$ and hence the Jacobi fields orthogonal to $X$ at $t=2L$ are exactly the ones that are orthogonal to $X$ at $t=0$.\vspace{4pt}
\item[(c)] The Jacobi fields that vanish in $(0,2L)$ belong to $\fk_+$. By  \lref{smooth} (a),  $X_1+Y_1$ and $X_3+Y_3$ are orthogonal to $X$ at $t=0$, while $X_2+Y_2$ is orthogonal to $X$ at $t=0$  by assumption.\vspace{4pt}
\item[(d)] Since $Z=p X_1+q Y_1$ is the only element of $V$ vanishing at $t=0$, we need to prove that $g_{c(0)}(X',Z')=0$.
By \lref{smooth} (c),  we have $g(X,Z)''(0)=0$ since  $p,q\ne 0$. Clearly $Z(0)=0$, and since $Z$ is a Jacobi field, $Z''(0)=-R_{c(0)}(Z,\dot c)\dot c=0$ as well. Thus we have at $t=0$:
$$
0=g(X,Z)''=g(X'',Z)+2g(X',Z')+g(X,Z'')=2g(X',Z')
$$
 \end{enumerate}

It follows that for some constants $a$ and $b$, the Killing vector field $a X_2+b Y_2$ is parallel  on the interval $[0,2L]$.

We now use the fact that if $J_1$ and $J_2$ are two Jacobi fields in a self adjoint family with $J_1$ parallel, then $g(J_1,J_2)'=g(J_1,J_2')=g(J_1',J_2)=0$. In particular, if $J_1$ is orthogonal to  $J_2$ at one point, they are orthogonal everywhere. Applying this to the parallel Jacobi field $a X_2+b Y_2$ it follows that it is orthogonal to
$ X_2+ Y_2$ everywhere since it is at $t=0$ by construction. Near $t=L$ we have $ (X_2^*+ Y_2^*)(c(L-t))=(L-t)( X_2+ Y_2)$ and hence $X$ is orthogonal to $X_2+ Y_2\in\fp$. But  the only vector in $\spam\{X_2,Y_2\}$  orthogonal to $ X_2+ Y_2$ is $ X_2-Y_2$   due to $K_+$ invariance. This proves our claim that $ X_2-Y_2$ is parallel.
\end{proof}

We can now finish the proof of the Theorem  in the Introduction. By \lref{parallel} we have that
$\ml X_2-Y_2,X_2-Y_2\mr=f_2-2h_2+g_2$ is constant. On the other hand, by \lref{smooth} (b),   the functions $f_2$ and $g_2$ are even near $t=0$ and $h_2$ is odd, and thus $h_2=0$. Now notice that this property in fact holds for all $t\in [-L,L]$. Indeed, the geodesic $c$ is a minimizing geodesic from the singular orbit at $t=0$ to the one at $t=L$. Since $G$ acts transitively on the orbits, and $K_-$ transitively on the unit sphere in $D_-$, it follows that the normal exponential map of the singular orbit $G/K_-$ is a diffeomorphism on a tubular neighborhood of radius $L$.   Thus the slice in the proof of the smoothness conditions can be defined on a ball of radius $L$.  Altogether, this implies that  $h_2(t)=0$ for $t\in [0,L)$. But then $\ml X_2+Y_2,X_2+Y_2\mr=f_2+g_2=\ml X_2-Y_2,X_2-Y_2\mr$ is constant as well, which is a contradiction since $X_2+Y_2$ must vanish at $t=L$.

 \bigskip

 \begin{rem*}
(a) For the second family of cohomogeneity one manifolds $E_{p,q}^*$  the principal isotropy group $H$ is trivial and hence, at  $t=0$,  $K_-$ acts as $R(\theta)$ on $D_-$. But this implies that $h_2+h_3=t^{2|p-q|}\phi_1(t^2)$ and $h_2-h_3=t^{2|p+q|}\phi_1(t^2)$, see the proof of \lref{smooth}. Hence $h_2$ is even and the last part of the above proof breaks down. Notice though that the methods imply that the metric is  restricted when it has nonnegative curvature: the fields $X_2^*-Y_2^*$ and $X_3^*-Y_3^*$ must be parallel.

(b)
Another difference between $E_{p,q}$ and $E_{p,q}^*$ is that in the latter case the Weyl group element $w_- $ is one of $ (\pm 1,\pm 1 )$ and hence  acts as $\Id$ on $\fp\oplus\fm$. Furthermore, $W\simeq\Z_2$, and hence $c$ has length $2L$.

(c) We finally observe that at least one of the spaces $E_{p,q}^*$ carries a metric with nonnegative curvature. Indeed,  we have an action by $\Sp(1)\Sp(1)$ on $\Sph^3\times\CP^2$   where $(q_1,q_2)\in \Sp(1)\Sp(1)$ acts via $(r,p)\to (q_1rq_2^{-1},q_2p)$. Here $r\in \Sph^3\simeq \Sp(1)$ and $p\in\CP^2$ with $q_2$ acting linearly via $\Sp(1)=\SU(2)\subset \SU(3)$. One now easily checks that the group diagram of this action is that of $E_{1,1}^*$, and hence they are equivariantly diffeomorphic.
 \end{rem*}

\bigskip

 \providecommand{\bysame}{\leavevmode\hbox
to3em{\hrulefill}\thinspace}

\end{document}